\documentclass[reqno]{amsart}
\usepackage{amsmath,amsfonts,amsthm,amssymb}
\usepackage[a4paper]{geometry}
\usepackage{hyperref}
\newtheorem{thm}{Theorem}
\newtheorem{lem}[thm]{Lemma}
\newtheorem{cor}[thm]{Corollary}
\theoremstyle{remark}
\newtheorem*{rem}{Remark}
\def\C{\mathbb C}
\def\R{\mathbb R}
\def\e{\mathrm e}
\def\d{\mathrm d}
\let\i\relax \def\i{\mathrm i}
\let\Re\relax
\DeclareMathOperator{\Re}{Re}

\DeclareMathOperator{\spec}{spec}
\DeclareMathOperator{\conv}{conv}
\DeclareMathOperator{\diag}{diag}
\DeclareMathOperator{\bdiag}{b-diag}

\begin{document}
\author{Jens Wirth}
\address{Jens Wirth, Department of Mathematics, Imperial College London, 180 Queen's Gate, London SW7 2AZ, UK}
\email{j.wirth@imperial.ac.uk}
\thanks{Author supported by EPSRC with grant EP/E062873/1.}
\title{Block-diagonalisation of matrices and operators}
\date{\today}
\begin{abstract}
In this short note we deal with a constructive scheme to decompose a continuous family of matrices 
$A(\rho)$ asymptotically as $\rho\to0$ into blocks corresponding to groups of eigenvalues of the limit matrix $A(0)$. We also discuss the extension of the scheme to matrix families depending upon additional parameters and operators on Hilbert spaces.  
\end{abstract}
\maketitle
\section{Matrix theory}
\subsection{Preliminaries}\label{sec:1.1} We first recall some well-known facts about matrix equations of 
Sylvester type and their solution. Let $A^+, A^-\in\C^{m\times m}$ be two matrices with
\begin{equation}\label{eq:1}
   \Re \spec A^+ > 0,\qquad \Re\spec A^- <0.
\end{equation}
Then a solution to the Sylvester equation
\begin{equation}\label{eq:2}
   A^+X - X A^- = B
\end{equation}
for a given right hand side $B\in\C^{m\times m}$ can be represented by the integral
\begin{equation}\label{eq:3}
X = \int _0^\infty \e^{-tA^+} B \e^{tA^-} \d t.
\end{equation}
Indeed, by assumption \eqref{eq:1} we know that there exists a constant $c>0$ such that the matrix exponentials satisfy  $\|\e^{-tA^+}\|\,\|\e^{tA^-}\|\lesssim \e^{-ct}$ and the integral converges exponentially. Furthermore, plugging \eqref{eq:3} into \eqref{eq:2} immediately yields
\begin{equation*}
 A^+X - X A^-= \int_0^\infty \big( A^+ \e^{-tA^+} B \e^{tA^-} - \e^{-tA^+} B \e^{tA^-}A^-\big) \d t = 
 - \int_0^\infty \frac{\d}{\d t} \big(\e^{-tA^+} B \e^{tA^-}\big) \d t = B.
\end{equation*}
The special case, where $A^-=-(A^+)^*$ is known as Lyapunov equation and plays an essential role in control theory. For some details see, e.g., \cite{Datta:1999}.

We will apply this representation in a slightly modified form. Assume for this that there exist numbers $\alpha\in\C^\times=\{\alpha\in\C\,:\,\alpha\ne0\}$ and $\beta\in\R$ such that 
\begin{equation}\label{eq:1'}
   \Re (\alpha \spec A^+) > \beta,\qquad \Re(\alpha\spec A^-) <\beta.\tag{1'}
\end{equation}
Then \eqref{eq:2} can be solved by an integral of the form \eqref{eq:3} replacing the path of integration by a suitable ray in the complex plane,
\begin{equation}\label{eq:3'}
X = \int _0^{\alpha\infty} \e^{-tA^+} B \e^{tA^-} \d t,\tag{3'}
\end{equation}
(or simply by dividing equation \eqref{eq:2} by $\alpha$). Condition \eqref{eq:1'} means that the spectra of $A^+$ and $A^-$ are separated by the line $\{ \zeta\in\C\, :\,  \Re(\alpha\zeta)=\beta \}$.  

\subsection{Matrix families and spectral block-decomposition} We proceed to our main topic and consider a family of matrices $A(\rho)\in\C^{m\times m}$ depending upon a (real or complex) parameter $\rho$. We assume it has a full asymptotic expansion as $\rho\to0$,
\begin{equation}\label{eq:4}
A(\rho) \sim A_0 + \rho A_1 + \rho^2 A_2 + \cdots, \quad \rho\to0,
\end{equation}
meaning that the difference of $A(\rho)$ and the first $N$ terms on the right is of order $\mathcal O(\rho^N)$. 

We assume further that the spectrum of the matrix $A_0$ is decomposed into groups of eigenvalues,
\begin{equation}\label{eq:5}
  \spec A_0 = \dot\bigcup_{j\in\mathcal I} \mathfrak S_j, \qquad \conv \mathfrak S_i\cap \conv \mathfrak S_j = \emptyset, \quad i\ne j ,  
\end{equation}
such that any two components $\mathfrak S_i$ and $\mathfrak S_j$ have separated convex hulls. Then the following statement holds true.

\begin{thm}\label{thm1}
Assume \eqref{eq:4} and \eqref{eq:5}.  Then there exists an invertible matrix family 
$M(\rho)$ having a full asymptotic expansion as $\rho\to0$ such that the matrix
\begin{equation}
M^{-1} (\rho)A(\rho) M(\rho)
\end{equation}
is block-diagonal modulo $\bigcap \mathcal O(\rho^N)$ with blocks corresponding to the groups of eigenvalues given in \eqref{eq:5}.
\end{thm} 

We will prove the following equivalent statement in a purely constructive way. The prove follows the standard scheme from \cite[Section 2.1]{Jachmann:2008a}.

\begin{thm}\label{thm2}
Assume \eqref{eq:4} and \eqref{eq:5}.  Then there exists for any number $N$ 
\begin{enumerate}
\item matrices $M_0, \ldots, M_{N-1} \in \C^{m\times m}$, $M_0$ invertible, and
\item matrices $\Lambda_0,\ldots, \Lambda_{N-1}$, block-diagonal with blocks corresponding to the partition of eigenvalues given in \eqref{eq:5},
\end{enumerate}
such that   
\begin{equation}
A(\rho) \left(\sum_{k=0}^{N-1} \rho^k M_k\right) - \left(
\sum_{k=0}^{N-1} \rho^k M_k \right)\left(\sum_{k=0}^{N-1}\rho^k \Lambda_k \right) = \mathcal O(\rho^N).
\end{equation}
\end{thm} 
\begin{proof} Without loss of generality we may assume that $A_0$ is already of block-diagonal form, $A_0=\bdiag(S_1,\ldots ,S_d)$ with $d=|\mathcal I|$ and $\spec S_j = \mathfrak S_j$. Then the corresponding statement with $N=1$ is valid with matrices $M_0=\mathrm I$ and $\Lambda_0=A_0$.

Now assume that the statement is already proven for a certain number $N=\ell$. We are going to construct the matrices $M_{\ell}$ and $\Lambda_{\ell}$ in such a way that the statement with $N=\ell+1$ follows. Since we assumed that $A(\rho)$ has a full asymptotic expansion the matrix family 
\begin{equation}
  B_{\ell} (\rho)  = A(\rho) \left(\sum_{k=0}^{\ell-1} \rho^k M_k\right) - \left(
\sum_{k=0}^{\ell-1} \rho^k M_k \right)\left(\sum_{k=0}^{\ell-1}\rho^k \Lambda_k \right) = \mathcal O(\rho^{\ell})
\end{equation} 
has a full asymptotic expansion. We denote its leading coefficient as $\tilde B_{\ell} = \lim_{\rho\to0} \rho^{-\ell}  B_\ell(\rho)$. Then we define $\Lambda_{\ell} = \bdiag \tilde B_{\ell}$, where $\bdiag$ selects the block-diagonal according to the partition~\eqref{eq:5}, and define $M_\ell$ to be a solution to the commutator equation
\begin{equation}\label{eq:9}
   [A_0,M_\ell]  + \tilde B_\ell - \Lambda_\ell = 0.
\end{equation}
Before we are going to solve \eqref{eq:9}, we conclude the induction argument. If we consider
the next approximation $B_{\ell+1}(\rho)$, we obtain
\begin{equation}
  B_{\ell+1}(\rho) = B_{\ell}(\rho) + \rho^\ell \big( [A_0,M_\ell] -\Lambda_\ell \big) + \mathcal O(\rho^{\ell+1}) = \mathcal O(\rho^{\ell+1})
\end{equation}
by definition of $\tilde B_\ell$ and \eqref{eq:9}. In order to solve \eqref{eq:9}, we write $M_\ell$ and $\tilde B_\ell$ as block matrices with respect to the partition \eqref{eq:5},
\begin{equation}
  M_\ell = \begin{pmatrix} M_\ell^{(1,1)} & \cdots & M_\ell^{(1,d)}\\\vdots &&\vdots \\ M_\ell^{(d,1)} & \cdots & M^{(d,d)}_\ell \end{pmatrix},\qquad
  \tilde B_\ell = \begin{pmatrix} \tilde B_\ell^{(1,1)} & \cdots & \tilde B_\ell^{(1,d)}\\\vdots &&\vdots \\ \tilde  B_\ell^{(d,1)} & \cdots & \tilde B^{(d,d)}_\ell \end{pmatrix},
\end{equation}
such that \eqref{eq:9} reads as system of equations $[S_j, M_\ell^{(j,j)}]=0$ and 
\begin{equation}\label{eq:12}
  S_i M_\ell^{(i,j)} - M_\ell^{(i,j)} S_j = - \tilde B_\ell^{(i,j)}.  
\end{equation}
Equation \eqref{eq:12} can be solved by the method already described in Section~\ref{sec:1.1}, while for the diagonal entries we may simply define $M_\ell^{(j,j)}=0$. 
\end{proof}

\begin{proof}[On the proof of Theorem~\ref{thm1}]
To conclude the proof of the first theorem we have to justify two things: First, to any sequence of matrices $M_k$ (or $\Lambda_k$ resp.) there exists a (smooth) family of matrices 
$M(\rho)$ (or $\Lambda(\rho)$ resp.) with the prescribed asymptotic expansion. This is a well-known fact from asymptotic analysis and a direct consequence of Borel's theorem. Second, because the set of invertible matrices is open within $\C^{m\times m}$, the invertibility of $M(0)=M_0$ implies invertibility of $M(\rho)$ for sufficiently small $\rho$. 
\end{proof}

\subsection{Perturbation theory and further additional parameters}\label{sec:1.3} Now we assume that the matrix families depend {\em continuously} on a further parameter $A(\rho,v)$, $v\in\Upsilon$, $\Upsilon$ being a compact metric space, such that the asymptotic expansions \eqref{eq:4} are {\em uniform} with respect to the parameters (and have coefficients $A_j(v)$, $v\in\Upsilon$, which are uniformly bounded with respect to the parameter $v$) and the separation condition \eqref{eq:5} is uniform in the parameter, i.e., there exists a constant $\delta$ such that between any two (parameter-dependent) groups of eigenvalues $\mathfrak S_i(v)$ and $\mathfrak S_j(v)$ a (parameter-dependent) closed separating strip\footnote{I.e. we assume that $\mathfrak S_i(v)$ and $\mathfrak S_j(v)$ are both disjoint to a strip $\{\zeta\in\C \;:\; |\Re (\alpha(v) \zeta - \beta(v) ) | \le \delta/2\}$ with suitable complex numbers $\alpha(v),\beta(v)\in\C$, $|\alpha(v)|=1$, depending continuously on the parameter $v\in\Upsilon$.} of width $\delta$  can be placed.

\begin{cor}\label{cor3}
Under these assumptions Theorems~\ref{thm1} and~\ref{thm2} hold uniform in the involved parameters $v\in\Upsilon$, especially all occuring matrices are uniformly bounded with respect to $v$. 
\end{cor}

\begin{proof}
It suffices to provide a-priori bounds on the matrices $M_\ell$ constructed within the proof of Theorem~\ref{thm2}. If we assume the above given uniform separation of the components $\mathfrak S_j(v)$ of the spectrum of $A_0$, \eqref{eq:3'} applied to \eqref{eq:12} yields 
\begin{equation}\label{eq:13}
  \| M_\ell^{(i,j)}(v) \| \le \int_0^\infty \| \e^{-t \alpha(v) S_i(v)} \tilde B_\ell^{(ij)}(v) \e^{t \alpha(v) S_j(v)} \| \d t \le 
 C(v) \|\tilde B_\ell^{(i,j)}(v)\|
\end{equation}
with constants $C(v)$ estimated via
\begin{equation}\label{eq:14}
 C(v) =  \int_0^\infty \| \e^{-t \alpha(v) S_i(v)}\|\,\|\e^{t \alpha(v) S_j(v)} \| \d t 
 \le \int_0^T \e^{t\|S_i(v)\|+t\|S_j(v)\|}\d t +2 \int_T^\infty \e^{-\delta t}\d t
\end{equation}
for $T$ choosen big enough. The choice of $T$ is based on the spectral radius formula, which implies that
\begin{equation}
\|\e^{ \pm t (\alpha(v)S_{i/j}(v)-\beta(v))}\|^{1/t} \le 2\, \mathrm r_{\spec}(\e^{\pm(\alpha(v)S_{i/j}(v)-\beta(v))}) =2 \e^{-\delta/2}
\end{equation} 
for all $t$ big enough. Because the left hand side is  continuous in $v$ and the estimate for fixed $t$ implies the same estimate for $2t$ by submultiplicativity of the matrix norm\footnote{$\|e^{tA}\|^{1/t} \le c$ implies $\|e^{2tA}\|^{1/t}= \|e^{tA}\e^{tA}\|^{1/t} \le \|e^{tA}\|^{2/t}\le c^2$ and thus $\|e^{2tA}\|^{1/(2t)}\le c.$} , the constant $T$ may be chosen locally uniform in $v$. By compactness of $\Upsilon$ the constants $C(v)$ are uniformly bounded. Following the inductive scheme, we obtain bounds for all matrices uniform in $v\in\Upsilon$. 
\end{proof}

\begin{rem}
This is of particular interest, if $A_0(v)$ has multiple eigenvalues for some $v\in\Upsilon$, because then Theorem~\ref{thm1} gives a constructive approach to separate the group of these degenerate eigenvalues from the remaining spectrum in a uniform way. 
\end{rem}

\begin{rem}
In particular cases the estimate \eqref{eq:13} can be improved. If we assume that the matrices $S_i(v)$ and $S_j(v)$ are both normal, the spectral radius of the semigroups coincides with their norm and hence the integral can be estimated directly to conclude $C(v)=\delta^{-1}$. If both matrices are uniformly diagonalisable, the estimate has the form $C(v)=C\delta^{-1}$, where $C$ is the product of the condition numbers of diagonalisers to $S_i(v)$ and $S_j(v)$. 
\end{rem}

\subsection{Multi-step schemes} In \cite[Section 2.2]{Jachmann:2008a}, \cite{Jachmann:2008} multi-step diagonalisation schemes have been considered. They are based on a similar block-diagonalisation, but of minimal block size (i.e., blocks correspond to single eigenvalues of $A_0$) and---based on the assumption that corresponding components of $A_0$ are diagonable---refined by applying the block-diagonalisation scheme to the terms of lower order inside these blocks. The scheme can be iterated and if a related hierarchy of conditions is satisfied, allows for perfect diagonalisation. 

We will not go into the particulars of this algorithm, but draw a consequence from its basic idea in combination with Theorem~\ref{thm1}. We do not need {\em any} particular assumptions besides the existence of the full asymptotic expansion \eqref{eq:4} in order to conclude the following normal form of the matrix family $A(\rho)$.

\begin{thm}\label{thm4}
Assume \eqref{eq:4}. Then there exists an invertible matrix family $M(\rho)$ having a full asymptotic expansion as $\rho\to0$ such that the matrix 
\begin{equation}
M^{-1} (\rho)A(\rho) M(\rho)
\end{equation}
is block-diagonal and each of its blocks is either of dimension 1 or of the form
\begin{equation}
p_k(\rho) \mathrm I + \rho^k J + \mathcal O(\rho^{k+1}), 
\end{equation}
where $p_k$ is a polynomial of degree at most $k$ and $J$ is a Jordan type matrix, i.e., where the only non-zero entries are 1 and located on the first upper side diagonal. 
\end{thm}

\begin{rem}
We do not claim that $J$ is a Jordan block, if different Jordan blocks to the same eigenvalue appear at a certain stage the scheme can not disentangle them any further.  
\end{rem}

\begin{proof} 
We apply Theorem~\ref{thm1} for the trivial partition of $\spec A_0$ into its elements. This defines a first family of diagonalisers $M_1(\rho)$ block-diagonalising $A(\rho)$ up to infinite order. Consider now one of the resulting blocks, which is of the form
\begin{equation}
  S_j + \rho \Lambda_1^{(j)} + \rho^2 \Lambda_2^{(j)} + \cdots, \rho\to0,
\end{equation}
with leading term $S_j$ satisfying $|\spec S_j| = 1$. We distinguish three cases:

\noindent{\sl First case.} The matrix $S_j$ is of dimension one.

\noindent{\sl Second case.}   If $S_j$ is not diagonalisable. Then we can choose an invertible matrix $M_2^{(j)}$ in such a way that it transforms $S_j$ into its Jordan normal form $\lambda_j I + J_j$. 

\noindent{\sl Third case.} If $S_j$ is diagonalisable, it immediately follows that $S_j=\lambda_j \mathrm I$ and we can apply Theorem~\ref{thm1} to $\rho ( \Lambda_1^{(j)}+\rho \Lambda_2^{(j)}+\cdots)$, which gives a second diagonaliser $M_2^{(j)}(\rho)$ transforming the lower order terms.

Collecting the matrices into $M_2(\rho)=\diag(M_2^{(1)}(\rho),\ldots, M_2^{(d)}(\rho))$, we obtain an invertible family resolving blocks one step further. In order to conclude the proof of Theorem~\ref{thm4}
we iterate the procedure in the third case, which will either reduce the size of blocks and thus terminate after {\em finitely} many steps or lead at some point to a matrix $f(\rho)\mathrm I$. There is no other possibility, because if in the third case a certain number of eigenvalues coincides up to infinite order the corresponding block  $S_j + \rho \Lambda_1^{(j)} + \rho^2 \Lambda_2^{(j)} + \cdots$ can only have multiples of the identity as coefficients (since the first non-identic one could not be diagonalisable otherwise).
\end{proof}

Even though the blocks have a relatively simple form, it is hard to draw any strong consequences on their spectrum from this representation. The spectrum of a small perturbation of large Jordan blocks and more generally Jordan type matrices can be quite far away from the spectrum of the Jordan matrix itself as \cite{Davies:2006} pointed out.

\section{Operator formulation}\label{sec:2}
We will finally give some remarks on similar approaches for operators on Banach or Hilbert spaces instead of finite-dimensional matrices. If all operators involved are bounded and we consider {\em finitely} many spectral components, the statements of Theorem~\ref{thm1} and~\ref{thm2} transfer immediately with the same proofs.  To decompose into infinitely many components appears to be problematic in general because then the constructed block operator matrix $M_1$ may become unbounded due to the non-uniformity of the decomposition \eqref{eq:5}. 

To obtain a non-trivial situation we consider a pencil of unbounded operators on a Hilbert space $\mathcal H$,
\begin{equation}
   A(\rho) = A_0 + \rho A_1,
\end{equation}
$A_0 : \mathcal H \supset \mathcal D(A_0) \to \mathcal H$ closed and self-adjoint such that 
$\mathcal D(A_0)$ becomes a Hilbert space with respect to the graph inner product, 
and $A_1:\mathcal D(A_0) \to \mathcal H$ bounded. Then the pencil $A(\rho)$ is closed for all $\rho$
with domain $\mathcal D(A_0)$.

For the following we assume that $A_0$ is invertible and that its spectrum has a positive and a negative component. Thus we can write $A_0$ as block-diagonal matrix $A_0=\diag(A^+,A^-)$ with a positive operator $A^+$ acting on the positive spectral subspace $\mathcal H^+$ and a negative operator $A^-$ acting on the negative spectral subspace $\mathcal H^-$. Furthermore, we write $A_1$ as block matrix 
\begin{equation}
A_1 = \begin{pmatrix} A_1^{++} & A_1^{+-}\\ A_1^{-+} & A_1^{--} \end{pmatrix}
\end{equation}
with components mapping $A_1^{+-} : \mathcal D(A^-) \to \mathcal H^+$ etc. In the special situation that 
$A_0$ is the Dirac operator (after applying the Fouldy-Wouthousen transformation to make it diagonal) and $A_1$ a self-adjoint perturbation of a particular structure the Douglas-Kroll-Hess scheme block-diagonalises the pencil $A(\rho)$, see e.g. \cite{Jansen:1989}, \cite{Siedentop:2006} or \cite{Langer:2001} for a discussion of a related approach.

Here we want to investigate how the proof of Theorem~\ref{thm2} can be generalised to the above situation and used to construct a suitable block-diagonaliser. We will distinguish two different scenarios. 

\subsection{Arbitrary perturbations and bounded diagonaliser} We will not make restrictions on $A_1$ here, thus we treat non-selfadjoint perturbations. In order to apply the scheme, we need to solve a Sylvester type problem
\begin{equation}\label{eq:12'}
   A^+X-XA^- = B
\end{equation}
for an arbitrary operator $B: \mathcal H^-\supset \mathcal D(A^-)\to\mathcal H^+$. Again we want to represent $X$ as integral,
\begin{equation}\label{eq:13'}
   X = \int_0^\infty \e^{-tA^+}B\e^{tA^-} \d t
\end{equation}
but now we have to explain the meaning of it. The semigroups $\e^{tA^-}$ and $\e^{-tA^+}$ are contractions (with norm estimated by $\e^{-\delta t}$ for a certain constant $\delta$ measuring the distance of $0$ from the spectrum of $A^+$ and $A^-$). The interpretation of the integral \eqref{eq:13'} depends on mapping properties of the operator $B$.

It seems natural to introduce the sequence of Hilbert spaces $\mathcal H^{(\pm)}_{\gamma}  = \mathcal D( |A^{(\pm)}|^\gamma)$ for $\gamma\ge 0$ endowed with the inner product 
\begin{equation}
   (x,y)_{\mathcal H^{(\pm)}_{\gamma}} =  (|A^{(\pm)}|^\gamma x,|A^{(\pm)}|^\gamma y)
\end{equation}
and describe mapping properties of $B$ as boundedness between such spaces. We use the notation
$\mathcal B_{\gamma_1,\gamma_2} :=\mathcal B(\mathcal H^{-}_{\gamma_1},\mathcal H^{+}_{\gamma_2})$ for the Banach space of bounded operators. Note that $|A^{(\pm)}|^\gamma$ gives an isometric isomorphism between $\mathcal H^{(\pm)}_\gamma$ and $\mathcal H^{(\pm)}$.

If $B\in\mathcal B_{0,0}$, then the integral \eqref{eq:13'} exists as
Bochner integral and defines a bounded operator $X\in\mathcal B_{0,0}$.
As we will see below, in the sense of an improper Riemann (-Bochner) integral\footnote{We use the suggestive notation $\int_{\to a}^{\to b}= \lim_{a'\searrow a,\,b'\nearrow b} \int_{a'}^{b'}$.},
\begin{equation}\label{eq:20}
 B = -\int_{\to 0}^{\infty} \frac{\d}{\d t} \left(\e^{-tA^+}B\e^{tA^-} \right)\d t 
 =  A^+ X -X A^-,
\end{equation}
such that $X:\mathcal D(A^-)\to\mathcal D(A^+)$ and \eqref{eq:12'} follows directly.  If $B\in \mathcal B_{1,0}$, then the integral exists as Bochner integral and
defines a bounded operator $X\in\mathcal B_{1,0}$, i.e., an in general unbounded operator $X:\mathcal H^-\supset \mathcal D(X)\to\mathcal H^+$, but further assumptions are needed to check \eqref{eq:12'} and to ensure the boundedness of $X$. A possible way around this is given in the next statement:

\begin{lem}
Assume that $B\in\mathcal B_{\gamma,0}$ for some $\gamma \in [0,1)$. Then \eqref{eq:13'} exists
as Bochner integral and defines $X\in\mathcal B_{s,s}$ for $s\in[0,1)$. Furthermore,  $X:\mathcal D(A^-)\to\mathcal D(A^+)$ and satisfies \eqref{eq:12'}. 
\end{lem}
\begin{proof}
By the spectral theorem for self-adjoint operators we know that
\begin{equation}
\| |A^\pm|^{\gamma} \e^{\mp tA^\pm} \|_{\mathcal B(\mathcal H^\pm)} \le \begin{cases} \delta^\gamma \e^{-\delta t},\qquad & t\ge \gamma/\delta ,\\ t^{-\gamma}  \mathrm e^{\gamma(1+\ln\gamma)},& t\le \gamma/\delta,\end{cases}
\end{equation}
for all $\gamma\ge 0$ and with $\delta=\inf\spec |A^-|=\inf\spec|A^+|$ (w.l.o.g.). Therefore, for $s,\gamma<1$
\begin{align}
   \|X\|_{\mathcal B_{s,s}} 
    &\le \int_0^\infty \| |A^+|^s \e^{-tA^+}\|_{\mathcal B(\mathcal H^+)}
    \| B \|_{\mathcal B_{\gamma,0}} \| |A^-|^{\gamma-s} \e^{ tA^-} \|_{\mathcal B(\mathcal H^-)}\d t \notag\\
  & \lesssim \| B\|_{\mathcal B_{\gamma,0}}   \left( \int_0^{1} t^{-\max(s,\gamma)} \d t + 
   \int_{1}^\infty \mathrm e^{-2\delta t} \d t \right) < \infty.
\end{align}
It remains to check \eqref{eq:12'}. Let $x\in\mathcal D(A^-)=\mathcal H^-_1$. Then
\begin{align}
   B x
   & = \lim_{t\to0} \e^{-tA^+}B \e^{tA^-} x =  -\int_{\to 0}^{\infty}   \frac{\d}{\d t} \left(\e^{-tA^+}B\e^{tA^-} x \right)\d t \notag\\
   & =  A^+ \left(\int_{\to 0}^{\infty}  \e^{-tA^+}B\e^{tA^-} \d t \right)x -  \left(\int_{\to 0}^{\infty}  \e^{-tA^+}B\e^{tA^-} \d t\right)  A^- x,
\end{align} 
where we used that for each $t>0$ the operator $\e^{\mp tA^\pm} \in \mathcal B ( \mathcal H_s^\pm,\mathcal H_\infty^\pm)$, $\mathcal H_\infty^\pm=\bigcap_{r\ge 0}\mathcal H_r^\pm$. Hence,
\begin{equation}
	A^+ \int_{\to 0}^{\infty}  \e^{-tA^+}B\e^{tA^-} \d t = B x + X A^- x \in \mathcal H^+
\end{equation}
and by closedness of $A^+$ we conclude that $Xx\in\mathcal D(A^+)=\mathcal H^+_1$ and the left hand side is equal to $A^+ X x$. Thus, \eqref{eq:12'} together with the mapping property $X:\mathcal D(A^-)\to\mathcal D(A^+)$ follows.
\end{proof}

\begin{rem}
It is straightforward to give explicit bounds on the norms $\|X\|_{\mathcal B_{0,0}}$ and $\|X\|_{\mathcal B_{\gamma,\gamma}}$. They can both be estimated by
\begin{align}
\max (  \|X\|_{\mathcal B_{0,0}}, \|X\|_{\mathcal B_{\gamma,\gamma}}\|  )
&\le \|B\|_{\mathcal B_{\gamma,0}} \left(\e^{\gamma(1+\ln\gamma)}\int_0^{\gamma/\delta} t^{-\gamma} \e^{-\delta t}\d t  + \delta^{\gamma} \int_{\gamma/\delta}^\infty \e^{-2\delta t}\d t\right)\notag\\
&\le \|B\|_{\mathcal B_{\gamma,0}}\delta^{\gamma-1} \left( \frac{\gamma\e^{\gamma} }{1-\gamma} +  \frac{\e^{-2\gamma}}2  \right).
\end{align}
Especially for $\gamma=0$ we get $\|X\|_{\mathcal B_{0,0}}\le (2\delta)^{-1} \|B\|_{\mathcal B_{0,0}}$
as for normal matrices, cf. the remarks after Corollary~\ref{cor3}. For $\gamma\to1$ or $\delta\to0$
the bound blows up. 
\end{rem}

\begin{thm}
Assume that for a certain $\gamma\in[0,1)$ we know  $A_1\in\mathcal B(\mathcal H_{\gamma},\mathcal H)$. Then there exists an invertible family $M(\rho)$ of bounded operators such that
\begin{equation}
   M^{-1}(\rho)A(\rho)M(\rho)
\end{equation}
is block-diagonal modulo $\bigcap\mathcal O(\rho^N)$ (in the operator-norm sense $M^{-1}(\rho)\mathcal H_1\to\mathcal H$). 
\end{thm}

\begin{proof}[Sketch of proof] The proof works analogously to the one of Theorem~\ref{thm2}, replacing all matrices by the corresponding operators. We have to make sure that the corner entries  $\tilde B_\ell^{(1,2)}$ and $\tilde B_\ell^{(2,1)}$ always have the right boundedness properties in order to apply
\eqref{eq:13'}. In the first step this is exactly the assumption we made. For the following ones we proceed by induction. If $B_\ell(\rho) \in\mathcal B(\mathcal H_{\gamma},\mathcal H)$, it follows that $\Lambda_\ell\in\mathcal B(\mathcal H_{\gamma},\mathcal H)$  and that the constructed diagonaliser $M_\ell$ is bounded $M_\ell\in\mathcal B(\mathcal H_s,\mathcal H_s)$, $s\in[0,1)$ and maps $\mathcal D(A)$ into itself. Looking at $B_{\ell+1}(\rho)$ we see that the non-vanishing additional terms are of the form
\begin{equation}
  B_{\ell+1}(\rho) - B_{\ell}(\rho) = \rho^{\ell+1}  \big( A_1 M_\ell - \sum_{k=1}^\ell M_{\ell+1-k} \Lambda_k\big)    + \mathcal O(\rho^{\ell+2}) \in\mathcal B(\mathcal H_{\gamma},\mathcal H)
\end{equation} 
and the desired mapping properties follow. 
\end{proof}

\begin{rem}
We know by construction that $M(\rho) : \mathcal H_1 \to \mathcal H_1$, however, we don't know whether it is bounded as operator between these spaces.
\end{rem}

\subsection{Selfadjoint perturbations and unitary diagonaliser} If we consider only self-adjoint perturbations it might be of interest to construct a unitary block-diagonaliser $M(\rho)$. 
In order to apply a perturbation series method we consider the Cayley transform $K(\rho)$ of $M(\rho)$,
\begin{equation}
M(\rho) = \frac{\mathrm I-\i K(\rho)}{\mathrm I+\i K(\rho)}, \qquad K(\rho) = \rho K_1 + \rho^2 K_2 + \cdots .
\end{equation}
If $M(\rho)\to\mathrm I$ in the norm sense as $\rho\to0$, then $K(\rho)$ will be a family of bounded operators (at least for small $\rho$), which can be achieved under the same assumptions on the perturbation as in the previous section.  

In order to avoid unnecessary repetitions we will only sketch the main difference to the previous considerations. Instead of a formulation in the spirit of Theorem~\ref{thm2} we multiply both sides with
$(\mathrm I+\i K(\rho))$ and $(\mathrm I-\i K(\rho))$, respectively, and require that
\begin{equation}
  \left(I+\i\sum_{k=1}^{\ell-1} \rho^k K_k \right) A(\rho)  \left(I-\i\sum_{k=1}^{\ell-1} \rho^k K_k \right)
  -   \left(I-\i\sum_{k=1}^{\ell-1} \rho^k K_k \right)   \left(\sum_{k=0}^{\ell-1} \rho^k \Lambda_k \right) \left(I+\i\sum_{k=1}^{\ell-1} \rho^k K_k \right) 
\end{equation}
is of order $\mathcal O(\rho^{\ell})$ for suitable self-adjoint and bounded operators $K_k$ (having the same mapping properties as the $M_k$ in the previous section) and symmetric  operators $\Lambda_j$.
 
Denoting the $\rho^\ell$-part of this expression as $\tilde B_\ell$, we conclude that the next terms are 
$\Lambda_\ell=\diag\tilde B_\ell$ and $K_\ell$ subject to
\begin{equation}
   2\i [K_\ell, A_0] = \tilde B_\ell - \Lambda_\ell.
\end{equation}
This is again an equation of form \eqref{eq:12'} for the corner entries of the block-matrix $K_\ell$ and the same procedure can be applied to its solution. Since $A^\pm$ are self-adjoint, it follows that the upper right corner entry is the adjoint of the lower left and self-adjointness of $K_\ell$ follows.  

\section{Concluding remarks}
Diagonalisation schemes of the form of Theorem~\ref{thm2} separating single eigenvalues and their generalisations have been useful for quite a few problems in the theory of hyperbolic and hyperbolic-parabolic coupled systems, as already pointed out and discussed in \cite[Section 3]{Jachmann:2008a}. Utilisations of these ideas can be found in \cite{Taylor:1975}, \cite{Yagdjian:1997}, \cite{Reissig:2000}, \cite{Reissig:2005a}, \cite{Wirth:2007b}, \cite{Wirth:2007c} to name just a few references. They are stable enough to adapt them to diagonalisation schemes within symbol classes and to study the evolution of linear systems with variable coefficients.

The author thinks that the generalisations to blocks demonstrated in this note will be of use for applications to more degenerate situations, especially due to the stability under perturbations following from the considerations in Section~\ref{sec:1.3}.  

The statements extent known facts from perturbation theory of matrices, cf. \cite{Kato:1980}, where analytic dependence of matrices upon parameters were treated and representations in terms of Dunford were integrals given.

The considerations in Section~\ref{sec:2} allow weaker assumptions on the perturbation $A_1$ compared to \cite{Langer:2001} or \cite{Kraus:2004}, where diagonally dominated operators require essentially $\gamma=1/2$. However, our approach gives only an asymptotic decoupling/diagonalisation of block operator matrices compared to exact formulas for block-diagonaliser in terms of a factorisation of the Schur complement associated to the block matrix in  \cite{Langer:2001}.

\bibliographystyle{alpha}

\begin{thebibliography}{KLT04}

\bibitem[Dat99]{Datta:1999}
B.~N. Datta.
\newblock Stability and inertia.
\newblock {\em Linear Algebra Appl.}, 302/303:563--600, 1999.

\bibitem[DH06]{Davies:2006}
E.~B. Davies and M.~Hager.
\newblock Perturbations of {J}ordan matrices.
\newblock Preprint, arXiv:math/0612158, 2006.

\bibitem[Jac08]{Jachmann:2008}
K.~Jachmann.
\newblock {\em A unified treatment of models of thermoelasticity}.
\newblock PhD thesis, TU Bergakademie Freiberg, 2008.

\bibitem[JH89]{Jansen:1989}
G.~Jansen and B.~A. Hess.
\newblock Revision of the {D}ouglas-{K}roll transformation.
\newblock {\em Physical Review A}, 39(11):6016--6017, 1989.

\bibitem[JW08]{Jachmann:2008a}
K.~Jachmann and J.~Wirth.
\newblock Diagonalisation schemes and applications.
\newblock Preprint, arXiv:0807.1009, 2008.

\bibitem[Kat80]{Kato:1980}
T.~Kato.
\newblock {\em {Perturbation theory for linear operators. Corr. printing of the
  2nd ed.}}
\newblock {Grundlehren der mathematischen Wissenschaften, 132.
  Berlin-Heidelberg-New York: Springer-Verlag.}, 1980.

\bibitem[KLT04]{Kraus:2004}
M.~Kraus, M.~Langer, and C.~Tretter.
\newblock {Variational principles and eigenvalue estimates for unbounded block
  operator matrices and applications.}
\newblock {\em J. Comput. Appl. Math.}, 171(1-2):311--334, 2004.

\bibitem[LT01]{Langer:2001}
Heinz Langer and Christiane Tretter.
\newblock {Diagonalization of certain block operator matrices and applications
  to Dirac operators.}
\newblock {Bart, H. (ed.) et al., Operator theory and analysis. The M. A.
  Kaashoek anniversary volume. Proceedings of the workshop, Amsterdam,
  Netherlands, November 12-14, 1997. Basel: Birkh\"auser. Oper. Theory, Adv.
  Appl. 122, 331-358 (2001).}, 2001.

\bibitem[RW05]{Reissig:2005a}
M.~Reissig and Y.-G. Wang.
\newblock {Cauchy problems for linear thermoelastic systems of type III in one
  space variable.}
\newblock {\em Math. Methods Appl. Sci.}, 28(11):1359--1381, 2005.

\bibitem[RW08]{Wirth:2007b}
M.~Reissig and J.~Wirth.
\newblock Anisotropic thermo-elasiticity in {2D} -- {P}art {I}: {A} unified
  treatment.
\newblock {\em Asymptot. Anal.}, 57(1-2):1--27, 2008.

\bibitem[RY00]{Reissig:2000}
M.~Reissig and K.~Yagdjian.
\newblock {$L\sb p$}-{$L\sb q$} decay estimates for the solutions of strictly
  hyperbolic equations of second order with increasing in time coefficients.
\newblock {\em Math. Nachr.}, 214:71--104, 2000.

\bibitem[SS06]{Siedentop:2006}
H.~Siedentop and E.~Stockmeyer.
\newblock The {D}ouglas-{K}roll-{H}e{\ss} method: Convergence and
  block-diagonalisation of {D}irac operators.
\newblock {\em Ann. Henri Poincar\'e}, 7:45--58, 2006.

\bibitem[Tay75]{Taylor:1975}
M.~E. Taylor.
\newblock {Reflection of singularities of solutions to systems of differential
  equations.}
\newblock {\em Commun. Pure Appl. Math.}, 28:457--478, 1975.

\bibitem[Wir08]{Wirth:2007c}
J.~Wirth.
\newblock Anisotropic thermo-elasiticity in {2D} -- {P}art {II}:
  {A}pplications.
\newblock {\em Asymptotic Anal.}, 57(1-2):29--40, 2008.

\bibitem[Yag97]{Yagdjian:1997}
K.~Yagdjian.
\newblock {\em The {C}auchy problem for hyperbolic operators}, volume~12 of
  {\em Mathematical Topics}.
\newblock Akademie Verlag, Berlin, 1997.

\end{thebibliography}

\end{document}